\documentclass[10pt, a4paper]{amsart}

\usepackage[french,english]{babel}
\usepackage[T1]{fontenc}
\usepackage[utf8]{inputenc}
\usepackage{amsmath, amssymb, amsfonts, amsthm}
\usepackage[margin=2.0cm]{geometry}
\usepackage[]{setspace}
\usepackage[,dvipsnames]{xcolor}
\usepackage[,unicode]{hyperref}
\hypersetup{
    urlcolor=blue,
    colorlinks=true,
   linkcolor= blue,
   citecolor=blue
}
 \usepackage{comment, stmaryrd}
\usepackage[all]{xy}
\usepackage{libertine,libertinust1math}
\usepackage{bbm}

\usepackage{enumerate}

\theoremstyle{plain}
\newtheorem{theorem}{Theorem}[section]
\newtheorem*{introtheorem}{Main Theorem}

\newtheorem{lemma}[theorem]{Lemma}
\newtheorem{corollary}[theorem]{Corollary}
\newtheorem{claim}{Claim}[theorem]

\newtheorem{observation}[theorem]{Observation}

\theoremstyle{definition}
\newtheorem{example}[theorem]{Example}
\newtheorem{definition}[theorem]{Definition}
\newtheorem{question}[theorem]{Question}

\theoremstyle{remark}
\newtheorem{remark}[theorem]{Remark}

\newenvironment{claimproof}[1][Proof of claim]
{%
	\proof[#1]%
}
{%
	\endproof%
}

\renewcommand\square\blacksquare
\newcommand\AKEsquaretall\blacklozenge

\renewcommand\eth{{\rm\textbf{(\dh)}}}

\newcommand{\bfk}{\mathbf{k}} 
\newcommand{\Lang}{\mathfrak{L}} 
\newcommand{\Lring}{\Lang_{\mathrm{ring}}} 
\newcommand{\Lval}{\Lang_{\val}} 
\newcommand\Th{\mathrm{Th}}

\newcommand\Sent{\mathrm{Sent}}

\newcommand\PP{\mathbb{P}}
\newcommand\FF{\mathbb{F}}

\newcommand\ZZ{\mathbb{Z}}
\newcommand\QQ{\mathbb{Q}}
\newcommand\RR{\mathbb{R}}
\newcommand\CC{\mathbb{C}}

\newcommand\ps[1]{#1(\!(t)\!)}
\newcommand\psvar[2]{#1(\!(#2)\!)}
\newcommand\ips[1]{#1[\![t]\!]}

\newcommand\Px{\mathrm{Px}}

\newcommand\val{\mathrm{val}}

\newcommand\bH{\mathbf{H}} 
\newcommand\bHen{\bH^{\mathrm{e}\prime}} 
\renewcommand\H{\mathrm{H}}
\newcommand\Hen{\H^{\mathrm{e}\prime}} 
\renewcommand\int{\mathbf{int}}

\newcommand\RfourF[1]{\hyperref[R4_F]{{\rm(R4)$_{#1}$}}}

\newcommand\hyp{$(\star)$}

\newcommand\define[1]{{\bf #1}}
\newcommand\resp[1]{{\rm[}#1{\rm ]}}
\numberwithin{equation}{section}
\newcommand\la\langle
\newcommand\ra\rangle

\newcommand\F{\mathbf{F}}
\newcommand\rmF{\mathrm{F}}

\newcommand\ACF{\mathbf{ACF}}

\newcommand\sfT{\mathsf{T}} 
\newcommand\sfS{\mathsf{S}} 
\newcommand\sfH{\mathsf{H}} 
\newcommand\sfHe{\sfH^{e}} 

\newcommand\alg{\mathrm{alg}}

\newcommand\triv{\mathrm{triv}}

\newcommand\E{\mathrm{E}}

\newcommand\SAcorrection[1]{#1}

\makeatletter
\newenvironment{abstracts}{%
	\footnotesize
  \newcommand{\abstractin}[1]{%
	\setlength{\parindent}{12pt}
    \selectlanguage{##1}%
    \item[\hskip\labelsep\scshape\abstractname.]%
  }%
  \global\setbox\abstractbox=\vbox \bgroup
    \hsize\textwidth \advance\hsize-8pc
    \list{}{\labelwidth\z@ \leftmargin3pc \rightmargin\leftmargin}%
}{%
  \endlist\egroup
}
\makeatother

\title[]{\Large\rm A note on existentially t-henselian fields}
\author{Sylvy Anscombe}
\thanks{\today}
\address{Universit\'{e} Paris Cit\'{e}, Sorbonne Universit\'{e}, CNRS, IMJ-PRG, F-75013 Paris, France}
\email{sylvy.anscombe@imj-prg.fr}

\begin{document}
\begin{abstracts}
\abstractin{english}
A field is existentially t-henselian if it is has the same existential theory in the first-order language of rings
as a field that admits a nontrivial henselian valuation.
This property turns out to be equivalent to $\ZZ$-largeness, which is a property identified in previous work with Fehm,
and which holds for $F$ if and only if $t\ips{F}$ is not Diophantine in $\ps{F}$, without extra constants.

In this short note, we further investigate this property in order to count the number of existential theories of henselian valuations on a given field, and to find other characterizations of existential t-henselianity.

\medskip

\abstractin{french}

Un corps est existentiellement t-hensélien s'il possède la même théorie existentielle, dans le langage des anneaux du premier ordre, qu'un corps admettant une valuation hensélienne non triviale.
Cette propriété s'avère équivalente à ``$\ZZ$-largeness'', une propriété identifiée dans mes travaux antérieurs avec Fehm,
et qui est vérifiée pour $F$ si et seulement si $\ips{F}$ n'est pas diophantienne dans $\ps{F}$,
sans constantes supplémentaires.

Dans cette brève note, nous approfondissons cette propriété afin d'énumérer les théories existentielles des valuations henséliennes sur un corps donné, et de trouver d'autres caractérisations de la propriété de t-henselianité existentielle
\end{abstracts}
\maketitle

In this note we define a field to be ``existentially t-henselian'' if it is has the same existential theory in the first-order language of rings
as a field that admits a nontrivial henselian valuation.
This property turns up naturally in the study of existential theories of equicharacteristic henselian valuations.
Perhaps more surprisingly, 
it turns out to be also equivalent to ``$\ZZ$-largeness'', another property identified in previous work with Fehm,
namely~\cite{AF17},
and which holds for a field $F$ if and only if $t\ips{F}$ is definable in $\ps{F}$ by an existential formula in the language of rings, without parameters.

We prove several small results on existentially t-henselian fields.
As an example, we are able to count the number of existential theories of henselian valuations on a given field:

\begin{introtheorem}\label{thm:1}
Let $K$ be any field.
There are at most three existential theories $\Th_{\exists}(K,v)$ in the language of valued fields, for equicharacteristic henselian valuations $v$ on $K$.
\end{introtheorem}

This theorem follows from Theorem~\ref{thm:1b}.

\section{Canonical valuations}

    For each field $K$, let 
	$\sfS_{\val}(K)$
	be the
	partially ordered set%
	\footnote{This is the Riemann--Zariski space of $K$ when suitably topologized.}
	of valuations of $K$,
	where the ordering corresponds to the usual notions of coarsenings and refinements of valuations.
	More precisely, 
	given valuations $v,w$ on $K$, we say that
$v$ is \define{finer} than $w$ (and that $w$ is \define{coarser} than $v$), denoted $v\leq w$,
if $\mathcal{O}_{v}\subseteq\mathcal{O}_{w}$.

	We note the following properties:
	\begin{enumerate}[{\bf(i)}]
	\item
		$\sfS_{\val}(K)$ is directed upwards, over sets,
		i.e.~every
		non-empty
		subset has a
       least upper bound. 
		In particular there is a maximum valuation, the trivial valuation, which is denoted $v_{K}^{\triv}$.
	\item
		The set of coarsenings of a given valuation is a chain, i.e.~a total order.
	\item
		Every chain has an infimum.
	\end{enumerate}
A valuation 
$v\in\sfS_{\val}(K)$ on $K$ is \define{henselian} if
for every finite extension $L/K$
there is precisely one $w\in\sfS_{\val}(L)$ that restricts to $v$.
Let $\sfH(K)$ be the partially ordered set of henselian valuations on $K$,
and we write $\sfH_{2}(K):=\{v\in\sfH(K)\mid\text{$Kv$ is separably closed}\}$
and $\sfH_{1}(K):=\sfH(K)\setminus\sfH_{2}(K)$.
	Then
	\begin{itemize}
	\item
		$\sfH_{2}(K)<\sfH_{1}(K)$, i.e.~every $v_{1}\in\sfH_{1}(K)$ is finer than every $v_{2}\in\sfH_{2}(K)$.
	\item
		$\sfH_{1}(K)$ is closed upwards,
		is a chain,
		and it has a minimum if $\sfH_{2}(K)$ is empty.
	\item
		$\sfH_{2}(K)$ is closed downwards,
		is not in general a chain,
		and it has a maximum if and only if it is nonempty.
	\end{itemize}
	The \define{canonical henselian valuation} $v_{K}$ is the supremum of $\sfH_{2}(K)$.
	Note that if $\sfH_{2}(K)\neq\emptyset$, then $v_{K}=\max\sfH_{2}(K)$;
	whereas if $\sfH_{2}(K)=\emptyset$, then $v_{K}=\min\sfH_{1}(K)$.
	A field $K$ itself is \define{henselian} if it admits a nontrivial henselian valuation,
	and it is \define{t-henselian} if some elementary extension $K^{*}\succeq K$ is henselian.
	It is a small part of the work of Jahnke and Koenigsmann,
	indeed it follows from Beth's Definability Theorem,
	that 
	if $Kv_{K}$ is not t-henselian, then $v_{K}$ is $\Lring$-definable without parameters.
	There are plenty of examples of t-henselian fields that are not henselian,
	including $\RR$ and $\FF_{p}^{\alg}$, and more besides.
	We observe also that henselianity (of a valuation) is $\Lval$-axiomatizable, but not finitely so.

	A similar game can be played with $p$-henselian valuations, those that extend uniquely to the $p$-closure, as was explored by Koenigsmann, and later by Jahnke and Koenigsmann together:

\begin{example}[{$p$-henselian valuations, Jahnke--Koenigsmann}]
Let $p\in\mathbb{P}$ be a prime number.
For a field $K$,
	we let $\sfH^{p}(K)$ be the
partially ordered set
	of $p$-henselian valuation rings on $K$.
Let
$\sfH^{p}(K)$
be the set of $p$-henselian valuations on $K$.
We write
$\sfH^{p}_{1}(K)=\{v\in\sfH^{p}(K)\mid Kv\neq Kv(p)\}$
and
$\sfH^{p}_{2}(K)=\{v\in\sfH^{p}(K)\mid Kv=Kv(p)\}$.
Then
	\begin{itemize}
	\item
	$\sfH_{2}^{p}(K)<\sfH_{1}^{p}(K)$,
	i.e.~every $v_{1}\in\sfH_{2}^{p}(K)$ is strictly finer than every $v_{2}\in\sfH_{1}^{p}(K)$.
	\item
	$\sfH_{1}^{p}(K)$ is closed upwards,
	is a chain,
	and it has a minimum if $\sfH_{2}(K)$ is empty.
	\item
	$\sfH_{2}^{p}(K)$ is closed downwards,
	in particular under minima of chains,
	but is not (in general) a chain nor closed upwards.
	\end{itemize}
In contrast to the situation with henselianity, $p$-henselianity is finitely axiomatizable.
and correspondingly membership of a valuation in each of $\sfH_{1}^{p}$ and $\sfH_{2}^{p}$ is finitely $\Lval$-axiomatizable.
The \define{canonical $p$-henselian valuation} $v_{K}^{p}$ is the supremum of $\sfH_{2}^{p}(K)$.
Note that if $\sfH_{2}^{p}(K)\neq\emptyset$, then $v_{K}^{p}=\max\sfH_{2}^{p}(K)$;
whereas if $\sfH_{2}^{p}(K)$, then $v_{K}^{p}=\min\sfH_{1}^{p}(K)$.
	As in the henselian case, if $Kv_{K}^{p}$ is not $p$-henselian then $v_{K}^{p}$ is $\Lring$-definable,
	by Beth's Definability Theorem.
\end{example}

\begin{example}[{Canonical tame valuation, Szewczyk}]
	Let $\sfT(K)$ be the partially ordered set of tame valuation rings on $K$,
	and let 
	$\sfT_{1}(K)=\sfT(K)\cap\sfH_{1}(K)$
	and
	$\sfT_{2}(K)=\sfT(K)\cap\sfH_{2}(K)$.
	As in the previous cases, there is an $\Lring$-elementary class of \define{t-tame} fields, which are those fields $k$ that are $\Lring$-elementarily equivalent to a field $k'$ admitting a nontrivial tame valuation.
	We denote by $v_{K}^{T}$
	the \define{canonical tame valuation}, which by definition is the supremum of $\sfT_{2}(K)$.
    This well-defines an element of $\sfT(K)$ by work of Szewczyk, see \cite{Szewczyk}.
	So if $\sfT_{2}(K)\neq\emptyset$, then $v_{K}^{T}=\max\sfT_{2}(K)$,
	whereas if $\sfT_{2}(K)=\emptyset$, then $v_{K}^{T}=\min\sfT_{1}(K)$.
	Just as in the henselian case, and in contrast to the $p$-henselian case, tameness is not finitely axiomatizable.
    The following result is immediate from Beth's Definability Theorem, applied in this context.
	\begin{theorem}
		If $Kv_{K}^{T}$ is not t-tame, then $v_{K}^{T}$ is $\Lring$-definable.
	\end{theorem}
    This theorem does not appear explicitly in either of
    {\cite{KRS,Szewczyk}}, but is closely related.
\end{example}

\begin{example}
		\define{Divisible-tame valuations}
		are those tame valuations which also have a divisible value group.
		The natural definition is that $k$ is \define{t-divisible-tame} if there exists $k^{*}\equiv k$ that admits a nontrivial divisible-tame valuation.
		It was shown in \cite[Lemma 5.1]{AJ} that if $k$ is t-divisible-tame then $k\preceq k((\QQ))$.
		It was also shown in \cite{AJ} that for each $p\in\mathbb{P}\cup\{0\}$, there exists a t-divisible-tame field $k$ of characteristic $p$ that is not henselian, not real closed, and not separably closed.
\end{example}

Similar observations hold for ``canonical valuations'' in many other related settings.

\begin{remark}
Many other properties of valuations that appear in the literature behave similarly.
The properties henselianity, $p$-henselianity, tameness\footnote{Tameness is not closed under composition, since the composition of a tame valuation of mixed characteristic $(0,p)$ with, for example, a discretely valued henselian valuation of equal characteristic zero is not tame. The closure of the class of tame valued fields under composition is the class of roughly tame fields.}, and divisible tameness, that appear in the preceding examples,
are well-behaved in several key respects.
To begin with, each of these properties is $\Lval$-axiomatizable, 
and most are closed under composition and decomposition.
Further examples of properties $P$ for valued fields $(K,v)$ satisfying these basic requirements are:
\begin{itemize}
\item
$v$ is algebraically maximal,
\item
$v$ is henselian and defectless,
\item
$v$ is roughly tame,
\item
the value group $vK$ is roughly $p$-divisible,
\item
the value group $vK$ is roughly divisible,
\item
$v$ is defectless and a coarsening of a given valuation $w$.
\end{itemize}
\end{remark}

We need not constrain our attention to $\Lval$-axiomatizability.
By working in a language of bi-valued fields, we may consider structures $(K,v,w)$ in which $v$ coarsens $w$.
Thus the chain $\sfH_{1}$ is replaced by the chain of coarsenings of each $w$.

\begin{example}
	By \cite[Lemma 3.4]{AJ24},
	if $(K,w)$ is NIP
	and of residue characteristic $p>0$,
	then $w$ has at most one coarsening $v$ with an imperfect residue field.
	Beth's Definability Theorem then implies that $v$ is definable in $(K,w)$.
There is also an explicit formula known, which is due to Scanlon.
\end{example}

All of the above are examples of the following theorem of Ketelsen, from her thesis:

\begin{theorem}[{Ketelsen, \cite[Theorem 4.2.4]{Ketelsen_thesis}}]\label{thm:Marga}
Let $P$ be a property of (possibly enriched) valued fields,
	and $P^{*}$ be a property of valued fields such that
	\begin{enumerate}[{\bf(i)}]
		\item
		$P$ is preserved under $\Lval'$-elementary equivalence, where $\Lval'=\Lang'\cup\{\mathcal{O}\}$ for some enrichment $\Lang'\supseteq\Lang$.
		\item
		$P^{*}$ is first-order axiomatizable in $\Lval$.
		\item
		For any field $F$ and valuations $u$ and $w$ on $F$ that have property $P$:
		if $u\leq w$ and $\bar{u}$ is the valuation induced by $u$ on $Fw$ then $(Fw,\bar{u})$ has $P^{*}$.
	\end{enumerate}
	Now let $(K,v)$ be such that $v\in\sfH_{1}(K)$, $(K,v)$ satisfies $P$ and $Kv$ is not t-henselian of type $P^{*}$.
	Then $\mathcal{O}_{v}$ is $\Lang'$-definable.
\end{theorem}

\section{Diophantine henselian valuation rings and valuation ideals}

In this section we consider fields $K$
and equicharacteristic henselian nontrivial valuations $v$ on $K$, with corresponding valuation rings $\mathcal{O}$.
The paper~\cite{FJ} gives a survey of the results around the subject of $\Lring$-definable henselian valuation rings.
\SAcorrection{A subset of a field $K$ is \define{Diophantine} if it is of the form
\begin{align*}
	\{x\in K\mid \exists y_{1},\ldots\exists y_{n}\in K\colon f_{1}(x,y_{1},\ldots,y_{n})=\ldots=f_{m}(x,y_{1},\ldots,y_{n})=0\},
\end{align*}
where $f_{1},\ldots,f_{m}\in\ZZ[X,Y_{1},\ldots,Y_{m}]$.
Equivalently, a subset of $K$ is Diophantine if it is defined by an existential $\Lring$-formula.
This notion of ``Diophantine'' disallows parameters, i.e.~the polynomials $f_{i}$ are over $\ZZ$.}

In \cite{AF17}, the following definition was introduced.

\begin{definition}[{\cite[Definition 3.5]{AF17}}]\label{def:ER_large}
	Let $\mathcal{C}$ be a class of fields.
	\begin{enumerate}[{\bf(i)}]
		\item
	We say that $\mathcal{C}$ has \define{embedded residue}
	if there exist $k_{1},k_{2}\in\mathcal{C}$, 
	a nontrivial valuation $v$ on $k_{1}$,
	and an embedding of rings $k_{1}v\rightarrow k_{2}$.
	We say that a field $k$ has \define{embedded residue} if $\{k'\mid k'\equiv k\}$ has embedded residue.
		\item
			We say that $\mathcal{C}$ is \define{$\ZZ$-large}\footnote{\SAcorrection{The paper~\cite{AF17} works in the context of ``$C$-fields'', that is fields $F$ equipped with a distinguished morphism $C\rightarrow F$, for a given integral domain $C$.
			In the present note we are always working simply in the case $C=\ZZ$.
			What we here call $\ZZ$-large is there simply called ``large''.}}
	if there exist $k_{1},k_{2}\in\mathcal{C}$, a subfield $E\subseteq k_{1}$,
	and a nontrivial henselian valuation $v$ on $E$ with $Ev\cong k_{2}$.
	We say that a field $k$ is \define{$\ZZ$-large} if $\{k'\mid k'\equiv k\}$ is $\ZZ$-large.
	\end{enumerate}
\end{definition}

There is an underlying duality between the
existential
definability of henselian valuation rings
and valuation ideals, which is somewhat twisted by the appearance of
henselianity
in the definition of $\ZZ$-largeness.
The following theorem of \cite{AF17} characterizes (among other things) when $\ips{F}$ \resp{respectively $t\ips{F}$} is Diophantine in $\ps{F}$.

\begin{theorem}[{\cite[Theorem 1.1]{AF17}}]\label{thm:intro_AF17}
Let $F$ be a field.
Then the following are equivalent.
	\begin{enumerate}[{\bf(i)}]
		\item
			There is an $\exists$-$\Lring$-formula that defines $\mathcal{O}_{v}$
			\resp{respectively, $\mathfrak{m}_{v}$}
			in $K$ for some equicharacteristic henselian nontrivially valued field
			$(K,v)$ with residue field $F$.
		\item
			There is an $\exists$-$\Lring$-formula that defines $\mathcal{O}_{v}$
			\resp{respectively, $\mathfrak{m}_{v}$}
			in $K$ for every henselian valued field
			$(K,v)$ with residue field elementarily equivalent to $F$.
		\item
			There is no elementary extension $F\preceq F^{*}$ with a nontrivial valuation $v$
			on $F^{*}$ for which the residue field $F^{*}v$ embeds into $F^{*}$
			\resp{respectively, with a nontrivial henselian valuation $v$ on
			a subfield $E$ of $F^{*}$ with $Ev\cong F^{*}$}.
	\end{enumerate}
\end{theorem}

\begin{remark}
	By \cite[Lemma 3.7]{AF17},
	condition {\bf(iii)} is equivalent to 
	\begin{enumerate}[{\bf(i)}]
			\setcounter{enumi}{3}
		\item[{\bf(iii')}]
			The field $F$ does not have embedded residue \resp{respectively, is not $\ZZ$-large}.
	\end{enumerate}
\end{remark}

We give several negative examples, i.e.~examples of $\mathcal{O}$ that are not Diophantine in $K$.

\begin{example}[{Negative examples}]
\
	\begin{enumerate}[{\bf(i)}]
	\item
		$\ips{\CC}$ is not Diophantine in $\ps{\CC}$:
		this folkloric result is explained in \cite[Appendix A]{AK14}.
	\item
		$\ips{\QQ_{p}}$ is not Diophantine in $\ps{\QQ_{p}}$:
		this is also explained in \cite[Appendix A]{AK14}.
	\item
		$\ips{\RR}$ is not Diophantine in $\ps{\RR}$: this is a similar direct limit argument.
			Any existential formula defining $\ips{\RR}$ in $\ps{\RR}$ must also define the (nontrivial) valuation ring $\ips{\RR}^{\Px}$ in the Puiseux series $\ps{\RR}^{\Px}$, but this is a real closed field.
	\end{enumerate}
\end{example}

These examples generalize to the following.

\begin{example}
	$\ips{\FF}$ is not Diophantine in $\ps{\FF}$, for any algebraically closed field $\FF$.
\begin{proof}
	To see this, we again give a direct limit argument:
	in the algebraic closure of $\ps{\FF}$
	the unique prolongation of $\ips{\FF}$ is again nontrivial and must be defined by any existential formula defining $\ips{\FF}$ in $\ps{\FF}$ since the algebraic closure is a direct limit of isomorphic copies of $\ps{\FF}$.
\end{proof}
\end{example}

The arguments here can be easily extended to show that $\ips{F}$ is not Diophantine in $\ps{F}$ whenever $F$ is a characteristic zero t-henselian field, i.e.~elementarily equivalent to one admitting a nontrivial henselian valuation.
In fact, as the main theorem will show, the characteristic assumption may be removed.

Turning to positive examples, we have the following.

	\begin{example}[{Positive examples}]
	\
	\begin{enumerate}[{\bf(i)}]
	\item
		$\ips{\FF_{q}}$ is Diophantine in $\ps{\FF_{q}}$ for all prime powers $q$ (\cite{AK14}).
	\item
		$\ips{F}$ is Diophantine in $\ps{F}$, for $F$ a PAC field not containing the algebraic closure of its prime subfield (\cite{Feh15}).
	\item
		$\ips{\QQ}$ is Diophantine in $\ps{\QQ}$ (\cite{AF17}).
	\end{enumerate}
	Each of these can be seen in a rather concrete fashion, with explicit formulas.
	\end{example}

\begin{remark}\label{rem:C}
Strictly speaking, the framework of \cite{AF17} was that of
\define{$C$-fields}: which are fields $F$ equipped with distinguished morphisms $C\rightarrow F$,
for a given integral domain $C$.
In the present note we are always working simply with the case $C=\ZZ$.
\end{remark}

\section{Existentially t-henselian fields}

For a field $K$,
we denote by
	\SAcorrection{$\sfHe(K)$}
the set of (equivalence classes of)
equicharacteristic henselian
valuations on $K$,
partially ordered by the relation of refinement/coarsening,
with largest element $v^{\rm triv}$, the trivial valuation on $K$.
Let $\E$ denote the $\Lring$-fragment $\Sent_{\exists}(\Lring)$.
Let $\bHen$ denote the $\Lval$-theory of equicharacteristic henselian nontrivially valued fields,
and let $\bHen(R)$ denote $\bHen$ together with axioms that impose that $R$ holds on the residue field, for any $\Lring$-theory $R$.
	\SAcorrection{%
The following theorem
from~\cite{AF26}---though it is based on~\cite{AF16}---%
will be used several times in this section.

	\begin{theorem}[{\cite[Corollary 3.19(a)(III)]{AF26}}]\label{thm:AF}
	Let $(K,v)\models\bHen$.
	Then
	$\Th_{\exists}(K,v)=\bHen(\Th(Kv))_{\exists}=\bHen(\Th_{\exists}(Kv))_{\exists}$.
	\end{theorem}%

	The actual statement of \cite[Corollary 3.19(a)(III)]{AF26}
	is
	$\Th_{\exists}(K,v)=\bHen(\Th_{\exists}(Kv))_{\exists}$,
	which yields the theorem when combined with the easy inclusions
	$\Th_{\exists}(K,v)\supseteq\bHen(\Th(Kv))_{\exists}\supseteq\bHen(\Th_{\exists}(Kv))_{\exists}$.

	\begin{remark}
		In fact
		there is a uniform version of Theorem~\ref{thm:AF}:
		for every $\Lring$-theory $R$
		we have
		$\bHen(R)_{\exists}=\bHen((R\cup\F)_{\exists})_{\exists}$,
		where $\F$ is the $\Lring$-theory of fields.
		This is not stated explicitly in \cite{AF26},
		although it is used in several examples,
		but follows from the ``in particular'' statement of \cite[Proposition 2.24]{AF26},
		which is given in an abstract setting.
	\end{remark}
	}

We say that $F$ satisfies \eth, or is \define{existentially t-henselian}\footnote{The letter \dh\ is pronounced ``eth''.}, if it satisfies the equivalent conditions in the following lemma,
\SAcorrection{which should be compared with~\cite[Section 6.3]{AF17}.}

\begin{lemma}\label{lem:eth_1}
For a field $F$,
the following are equivalent.
\begin{enumerate}[{\bf(i)}]
\item
$\Th_{\exists}(F)=\Th_{\exists}(\ps{F})$.
\item
$\Th_{\exists}(F)=\bHen(\Th(F))_{\E}$.
\item
$\Th_{\exists}(F)=\bHen(\Th_{\exists}(F))_{\E}$.
\item
$\Th_{\exists}(F)$ is a fixed point of the map
$T\longmapsto\bHen(T)_{\E}$
from the power set of $\E$ to itself.
\item
There exists a henselian field $F'$ such that
$\Th_{\exists}(F)=\Th_{\exists}(F')$.
\item
$F$ is $\ZZ$-large.
\item
$\{F'\mid\Th_{\exists}(F)=\Th_{\exists}(F')\}$ is $\ZZ$-large.
\end{enumerate}
\end{lemma}
\begin{proof}
	\SAcorrection{
		From Theorem~\ref{thm:AF} we have
		$\Th_{\exists}(\ps{F},v_{t})=\bHen(\Th(F))_{\exists}=\bHen(\Th_{\exists}(F))_{\exists}$.
		Taking the reduct to $\Lring$, this yields
		$\Th_{\exists}(\ps{F})=\bHen(\Th(F))_{\E}=\bHen(\Th_{\exists}(F))_{\E}$,
		and 
	}
thus
	\SAcorrection{{\bf(i,ii,iii,iv)}}
are equivalent.
If there exists
	\SAcorrection{$v\in\sfHe(F)$}
non-trivial,
then
$\Th_{\exists}(F,v)=\Th_{\exists}(\ps{F},v\circ v_{t})$,
and so
$\Th_{\exists}(F)=\bHen(\Th(F))_{\E}$.
Next suppose that there exists henselian $F'$ such that $\Th_{\exists}(F)=\Th_{\exists}(F')$.
Replacing $F'$ with an elementary extension if necessary, we may assume
	\SAcorrection{that}
there exists
	\SAcorrection{$v\in\sfHe(F')$}
non-trivial.
Then
$\Th_{\exists}(F)=\Th_{\exists}(F')=\bHen(\Th(F'))_{\E}=\bHen(\Th(F))_{\E}$,
which shows that
	\SAcorrection{{\bf(v)$\Rightarrow$(ii).}}
Conversely, if
$\Th_{\exists}(F)=\bHen(\Th(F))_{\E}$
then
$\Th_{\exists}(F)=\Th_{\exists}(\ps{F})$.
Thus
	\SAcorrection{{\bf(v)$\Rightarrow$(i).}}
By
\cite[Proposition 6.10 ($1\Leftrightarrow4$)]{AF17},
$\Th_{\exists}(F)=\Th_{\exists}(\ps{F})$
is equivalent to $\ZZ$-largeness,
which proves
	\SAcorrection{{\bf(i)$\Leftrightarrow$(vi).}}
Clearly
	\SAcorrection{{\bf(vi)$\Rightarrow$(vii).}}

Suppose that for some $F_{1},F_{2}\in\{F'\mid\Th_{\exists}(F)=\Th_{\exists}(F')\}$
there exists $F_{0}\subseteq F_{2}$ and a
non-trivial
henselian valuation $w$ on $F_{0}$ with $F_{0}w$ isomorphic to $F_{1}$.
Then $w$ is equicharacteristic and henselian, so we may embed $F_{1}(t)^{h}$ into $F_{2}$.
This shows that
$\Th_{\exists}(F_{1})\subseteq\Th_{\exists}(F_{1}(t)^{h})\subseteq\Th_{\exists}(F_{2})$.
By
\cite[Proposition 6.10 ($4\Rightarrow1$)]{AF17},
we conclude that
$F$ is $\ZZ$-large,
i.e.~%
	\SAcorrection{{\bf(vii)$\Rightarrow$(vi).}}
\end{proof}

\begin{lemma}\label{lem:eth_2}
Let $K$ be a field and
	\SAcorrection{$v\in\sfHe(K)$}
non-trivial.
Then
$Kv$ satisfies \eth\
if and only if
$\Th_{\exists}(K,v)=\bHen(\Th(K))_{\exists}$.
\end{lemma}
\begin{proof}
For the implication $\Rightarrow$ we have
	\SAcorrection{%
		\begin{align*}\begin{array}{rll}
		\Th_{\exists}(K,v)
		&=
		\bHen(\Th_{\exists}(Kv))_{\exists}
		&\text{by Theorem~\ref{thm:AF}}
		\\
		&=
		\bHen(\bHen(\Th_{\exists}(Kv))_{\E})_{\exists}
		&\text{by Lemma~\ref{lem:eth_1}{\bf(iii)} since $Kv$ satisfies \eth}
		\\
		&=
		\bHen(\Th_{\exists}(K))_{\exists}
		&\text{by Theorem~\ref{thm:AF}}
		\\
		&=
		\bHen(\Th(K))_{\exists},
		&\text{by Theorem~\ref{thm:AF}.}
	\end{array}\end{align*}%
	}

	\noindent
For the converse, from the hypothesis $\Th_{\exists}(K,v)=\bHen(\Th(K))_{\exists}$
it follows that $\Th_{\exists}(Kv)=\Th_{\exists}(K)$, and $K$ admits a
non-trivial
henselian valuation,
thus $Kv$ satisfies \eth.
\end{proof}

Note that any field with an algebraically closed subfield satisfies \eth.
We are now in a position to describe the theories
$\Th_{\exists}(K,v)$
for all
	\SAcorrection{$v\in\sfHe(K)$.}

\begin{theorem}[{Dichotomy}]\label{thm:Dichotomy}
Let
$K$ be a field of characteristic
$p\in\PP\cup\{0\}$.
We have the following dichotomy:
\begin{enumerate}[{\bf(i)}]
\item
Either
		\SAcorrection{$\sfHe(K)$}
is not linearly ordered,
in which case
$\FF_{p}^{\mathrm{alg}}\subseteq K$, and
for all
non-trivial
		\SAcorrection{$v\in\sfHe(K)$}
we have that
$Kv$ satisfies \eth\ and
$$\Th_{\exists}(K,v)=\bHen(\Th(K))_{\exists}=\bHen(\ACF_{p})_{\exists}.$$
\item
Or
		\SAcorrection{$\sfHe(K)$}
is linearly ordered,
in which case
it admits a smallest element $v^{e}_{K}$,
and for all
non-trivial
		\SAcorrection{$v\in\sfHe(K)\setminus\{v^{e}_{K}\}$}
we have
		\SAcorrection{that}
$Kv$ satisfies \eth\ and
$$\Th_{\exists}(K,v)=\bHen(\Th(K))_{\exists}=\bHen(\bHen(\Th(Kv^{e}_{K}))_{\Sent(\Lring)})_{\exists}.$$
Moreover if $v^{e}_{K}$ is
non-trivial
then $\Th_{\exists}(K,v^{e}_{K})=\bHen(\Th(Kv^{e}_{K}))_{\exists}$,
and this equals $\bHen(\Th(K))_{\exists}$ if and only if $Kv^{e}_{K}$ satisfies \eth.
\end{enumerate}
In particular, $Kv$ does not satisfy \eth\ for at most one
	\SAcorrection{$v\in\sfHe(K)$.}
If such a $v$ exists, it is the finest element $v^{e}_{K}$ of
	\SAcorrection{$\sfHe(K)$,}
which is totally ordered.
\end{theorem}

\begin{proof}
	\SAcorrection{%
Suppose first that
	$\sfHe(K)$
is not linearly ordered.
In this case there exist equicharacteristic henselian valuations $v$ on $K$
with $Kv$ separably closed,
for example the canonical henselian valuation $v_{K}$.
For notational simplicity we denote $\FF_{0}=\QQ$.
	\begin{claim}
		For every
		$v\in\sfHe(K)$,
	the residue field $Kv$ contains $\FF_{p}^{\alg}$.
	\end{claim}
	\begin{claimproof}
		Each
		$v\in\sfHe(K)$
		is either a refinement of $v_{K}$,
		in which case the residue field $Kv$ is also separably closed
		and the claim follows,
		or it is a proper coarsening of $v_{K}$.
		In this latter case, $v_{K}$ incudes a nontrivial henselian valuation $\bar{v}_{K}$ on $Kv$.
		The residue field $Kv_{K}$ is separably closed, and so $\FF_{p}^{\alg}\subseteq Kv_{K}$.
		By henselianity of $\bar{v}_{K}$
		(see for example~\cite[Lemma 2.3]{ADF23}),
		there is a partial section
		$\zeta:\FF_{p}^{\alg}\rightarrow Kv$
		of the residue map of $\bar{v}_{K}$,
		which proves the claim.
	\end{claimproof}
	Therefore
	$\Th_{\exists}(Kv)=\ACF_{p,\exists}$,
	and so $Kv$ satisfies~\eth.
	Next we suppose that $v$ is nontrivial,
	so
	$\Th_{\exists}(K,v)=\bHen(\Th(K))_{\exists}$
	by Lemma~\ref{lem:eth_2}.
	By Theorem~\ref{thm:AF} we have
	\begin{align*}
		\Th_{\exists}(K,v)=\bHen(\Th_{\exists}(Kv))_{\exists}=\bHen(\ACF_{p,\exists})_{\exists}=\bHen(\ACF_{p})_{\exists},
	\end{align*}
	which finishes case {\bf(i)}.}
	
Suppose that
	\SAcorrection{$\sfHe(K)$}
is linearly ordered.
Let $v^{e}_{K}$ denote the finest element of
	\SAcorrection{$\sfHe(K)$:}
the valuation ring corresponding to $v^{e}_{K}$ is the intersection of all the valuations rings of elements of
	\SAcorrection{$\sfHe(K)$,}
and such intersections of chains of equicharacteristic henselian valuation rings are again equicharacteristic henselian valuation rings.
For every
non-trivial
	\SAcorrection{$v\in\sfHe(K)\setminus\{v^{e}_{K}\}$, $v^{e}_{K}$}
induces a
non-trivial
equicharacteristic henselian valuation $\bar{v}^{e}_{K}$ on $Kv$.
Straight away
\SAcorrection{by Theorem~\ref{thm:AF}}
we have
$\Th_{\exists}(K,v)=\bHen(\Th(Kv))_{\exists}$
and
$\Th_{\exists}(Kv,\bar{v}^{e}_{K})=\bHen(\Th(Kv^{e}_{K}))_{\exists}$,
\SAcorrection{so that
$\Th_{\exists}(Kv)=\bHen(\Th(Kv^{e}_{K}))_{\E}$.}
Combining these equalities
\SAcorrection{and again applying Theorem~\ref{thm:AF},}
we have
$\Th_{\exists}(K,v)=
\bHen(\bHen(\Th(Kv^{e}_{K}))_{\E})_{\exists}=
\bHen(\bHen(\Th(Kv^{e}_{K}))_{\Sent(\Lring)})_{\exists}$.
Moreover
$Kv$ satisfies \eth, so
$\Th_{\exists}(K,v)=\bHen(\Th(K))_{\exists}$
by Lemma~\ref{lem:eth_2}.

Finally we suppose that $v^{e}_{K}$ is
non-trivial.
\SAcorrection{By a further application of Theorem~\ref{thm:AF}}
we have
$\Th_{\exists}(K,v^{e}_{K})=\bHen(\Th(Kv^{e}_{K}))_{\exists}$.
It only remains to argue that
$Kv^{e}_{K}$ satisfies \eth\
if and only if
$\Th_{\exists}(K,v^{e}_{K})=\bHen(\Th(K))_{\exists}$,
but this is simply Lemma~\ref{lem:eth_2} applied to $v=v^{e}_{K}$.
\end{proof}

\SAcorrection{The \define{canonical equicharacteristic henselian valuation} $v_{K}^{e}$, which was defined above in case {\bf(i)}, may also be defined in case {\bf(ii)} by writing $v_{K}^{e}:=v_{K}$.}
	\SAcorrection{This allows the rephrasing of Theorem~\ref{thm:Dichotomy}}
as follows:

\begin{theorem}\label{thm:1b}
Let
	$K$ be any field.
	For all
	non-trivial
	\SAcorrection{$v\in\sfHe(K)\setminus\{v^{e}_{K}\}$,}
	the valued fields $(K,v)$ share the same existential theory.
	Namely:
	$$\Th_{\exists}(K,v)=\bHen(\Th(K))_{\exists}=\bHen(
	\underbrace{\bHen(\Th(k))^{\vdash}\cap\Sent(\mathfrak{L}_{\mathrm{ring}})}%
	_{\text{the deductive closure in $\Lring$ of $\bHen(\Th(k))$}})_{\exists},$$
	where $k$ denotes the residue field of $v^{e}_{K}$.
\end{theorem}

This theorem
\SAcorrection{has the following corollary.}

\begin{corollary}
For any field $K$, there are at most three distinct existential theories of the valued fields
$(K,v)$,
for
	\SAcorrection{$v\in\sfHe(K)$}%
	:
\begin{enumerate}[{\bf(I)}]
\item
	the trivial case
		\SAcorrection{$v=v_{K}^{\triv}$}%
		,
\item
$v$ is
non-trivial
and $Kv$ satisfies \eth,
\item
$v$ is
non-trivial
and $Kv$ does not satisfy \eth.
\end{enumerate}
Cases {\bf(I,II,III)} are always distinct.
Case {\bf(I)} exists for every $K$, but both cases {\bf(II)} and {\bf(III)} may be void, independently, according to the following examples.
However, since at most one element of
	\SAcorrection{$\sfHe(K)$}
satisfies {\bf(III)},
case {\bf(II)} is void if and only if either
	\SAcorrection{$\sfHe(K)=\{v^{\triv}_{K}\}$}
or
	\SAcorrection{$\sfHe(K)=\{v^{\triv}_{K},v^{e}_{K}\}$}
with $Kv^{e}_{K}$ not satisfying \eth.
\end{corollary}

The Main Theorem is also a corollary of these results.

\begin{example}
\phantom{\quad}
\begin{enumerate}[$(a)$]
\item
Let $K=\QQ$.
There are no
non-trivial
elements of
	\SAcorrection{$\sfHe(\QQ)$.}
In this case $\QQ$ does not satisfy \eth.
Both
		\SAcorrection{{\bf(II)}}
and
		\SAcorrection{{\bf(III)}}
are void.
\item
Let $K=\QQ^{\mathrm{alg}}$.
There are no
non-trivial
elements of
	\SAcorrection{$\sfHe(\QQ^{\mathrm{alg}})$.}
In this case $\QQ^{\mathrm{alg}}$ satisfies \eth.
Nevertheless, both
		\SAcorrection{{\bf(II)}}
and
		\SAcorrection{{\bf(III)}}
are void.
\item
Let $K=\ps{\QQ}$.
The only
non-trivial
element of
	\SAcorrection{$\sfHe(\ps{\QQ})$}
is $v_{t}$,
which is in case
		\SAcorrection{{\bf(III)}}
since $\QQ$ does not satisfy \eth.
Thus
		\SAcorrection{{\bf(II)}}
is void.
\item
Let $K=\ps{\QQ^{\mathrm{alg}}}$.
The only
non-trivial
element of
	\SAcorrection{$\sfHe(\ps{\QQ^{\mathrm{alg}}})$}
is $v_{t}$,
which is in case
		\SAcorrection{{\bf(II)}}
since $\QQ^{\mathrm{alg}}$ satisfies \eth.
Thus
		\SAcorrection{{\bf(III)}}
is void.
\item
Let $K=\ps{\psvar{\QQ}{s}}$.
		\SAcorrection{
There are two
non-trivial
elements of
	\SAcorrection{$\sfHe(\ps{\psvar{\QQ}{s}})$,}
namely $v_{t}$
and $v_{s}\circ v_{t}$.
The former, the $t$-adic valuation $v_{t}$ has residue field $\psvar{\QQ}{s}$,
which is is henselian, so in particular \eth,
thus
$v_{t}$ is in case
		{\bf(II)}.
The latter, the composition $v_{s}\circ v_{t}$,
has residue field $\QQ$ which is in case
		{\bf(III)}.
		}
Note that $\QQ$ and $\psvar{\QQ}{s}$ have different existential theories,
thus so do
$(\ps{\psvar{\QQ}{s}},v_{t})$
and
$(\ps{\psvar{\QQ}{s}},v_{s}\circ v_{t})$.
\end{enumerate}
\end{example}

\section{Towards a more general perspective}

\SAcorrection{%
In \cite[Section 2]{AF26}, the author, together with Fehm, introduced a framework of contexts, bridges, and arches, in order to systematically study computable interpretations between fragments of theories.}
Let $A=(B,\hat{B},\iota)$ be an arch.
Suppose
\hyp\
that
$\Lang_{1}\subseteq\Lang_{2}$,
$\hat{L}_{1}=\Sent(\Lang_{1})\cap\hat{L}_{2}$,
and $L_{1}=\Sent(\Lang_{1})\cap L_{2}$.
We do not suppose $\iota$ to be the identity map.
The set of $L_{1}$-deductively closed $L_{1}$-theories forms a complete lattice.
Consider the map from
$\Lang_{1}$-theories to $L_{1}$-theories
given by
$R\mapsto T_{2}(R)_{L_{1}}$,
and let $\Phi_{A}$ denote the restriction of this map to $L_{1}$-theories.

\begin{definition}\label{def:fixed_point}
We say that $\Phi_{A}$ is \define{increasing} if
$R\subseteq\Phi_{A}(R)$,
for all $L_{1}$-theories $R$;
and that $\Phi_{A}$ is \define{idempotent} if
$\Phi_{A}(R)=\Phi_{A}\circ\Phi_{A}(R)$,
for all $R$.
\end{definition}

It is easy to see that $R\subseteq S$ implies $\Phi_{A}(R)\subseteq\Phi_{A}(S)$.
Now we consider the arch
$A=\rmF_{\exists}/\Hen_{\exists}\parallel\rmF/\Hen$,
\SAcorrection{using the notation of \cite{AF26},
so that
$\Phi_{A}$ is now the map
\begin{align*}
	\Phi_{A}(R)=\bHen(R)_{\E}.
\end{align*}}%
Note that the hypotheses \hyp\ are satisfied for $A$.
We continue to denote
$\E=L_{1}=\Sent_{\exists}(\Lring)$
and note that in this case $\iota=\iota_{\bfk}$ is not the identity map.
By Lemma~\ref{lem:eth_1},
a theory $\Th_{\exists}(k)$ is a fixed point of $\Phi_{A}$ if and only if $k$ satisfies $\eth$.

It is easy to see that
$\Phi_{A}$ is increasing,
at least when
$R=\Th_{\exists}(k)$
is the existential theory of a field:
since $k\subseteq\ps{k}$
and $\Th_{\exists}(\ps{k})=\bHen(\Th_{\exists}(k))_{\E}$,
we have
$\Th_{\exists}(k)\subseteq\bHen(\Th_{\exists}(k))_{\E}$.

\begin{lemma}
$\Phi_{A}$ is increasing.
\end{lemma}
\begin{proof}
Let $R$ be an $\E$-theory,
so that $R=R_{\exists}$.
We want $R\subseteq\bHen(R)_{\E}$,
i.e.~$\bHen(R)\models R$.
Let $(K,v)\models\bHen(R)$.
Then $Kv\models R$.
By \cite{ADF23} there exists an elementary extension
$(K,v)\preceq(K^{*},v^{*})$
and a partial section
$Kv\rightarrow K^{*}$
of the residue map of $v^{*}$.
Thus
$\Th_{\exists}(Kv)\subseteq\Th_{\exists}(K^{*},v^{*})$.
Thus $(K^{*},v^{*})\models R$.
\end{proof}

\newcommand\Sring{\Sent(\Lring)}

\begin{lemma}
$\bHen(R)_{\Sring}=\bHen(\bHen(R)_{\Sring})_{\Sring}$
for each $\Lring$-theory $R$.
\end{lemma}
\begin{proof}
Since $\Phi_{A}$ is increasing we have
$R\subseteq\bHen(R)_{\Sring}$.
Thus
$\bHen(R)\subseteq\bHen(\bHen(R)_{\Sring})$.
Let $(K,v)\models\bHen(R)$,
so that $Kv\models R$.
There exists an elementary extension
$(K,v)\preceq(K^{*},v^{*})$
such that $v^{*}$ admits a nontrivial proper coarsening $w$.
Then
$K^{*}v^{*}\models R$, so
	\SAcorrection{$(K^{*}w,\bar{v})\models\bHen(R)$}
and
$K^{*}w\models\bHen(R)_{\Sring}$.
Therefore
$(K^{*},w)\models\bHen(\bHen(R)_{\Sring})$,
and so
$K^{*}\models\bHen(\bHen(R)_{\Sring})_{\Sring}$.
This proves that
$\bHen(R)\models\bHen(\bHen(R)_{\Sring})_{\Sring}$.
\end{proof}

\begin{lemma}
$\Phi_{A}$ is idempotent.
\end{lemma}
\begin{proof}
Let $R$ be an $\E$-theory.
Taking the existential consequences of the previous lemma we have
$\bHen(R)_{\E}=\bHen(\bHen(R)_{\Sring})_{\E}$.
The latter is equal to $\bHen(\bHen(R)_{\E})_{\E}$.
\end{proof}

\begin{observation}
Let $R$ be an $\E$-theory of fields.
The following are equivalent.
\begin{enumerate}[{\bf(a)}]
	\item
	$R$ is a fixed point of $\Phi_{A}$.
	\item
	$R$ is in the image of $\Phi_{A}$.
	\item
	Every model of $R$ is $\ZZ$-large and $R=R_{\exists}$.
\end{enumerate}
\end{observation}
\begin{proof}
The equivalence
	\SAcorrection{{\bf(a)$\Leftrightarrow$\bf(b)}}
is trivial since $\Phi_{A}$ is idempotent.
For
	\SAcorrection{{\bf(c)$\Rightarrow$(a)}}%
	:
by monotonicity we have $R\subseteq\bHen(R)_{\E}$.
If $F\models R$ then $F$ is $\ZZ$-large, so $\Th_{\exists}(F)=\bHen(\Th_{\exists}(F))_{\E}$ by Lemma~\ref{lem:eth_1}.
Therefore
$\Th_{\exists}(F)=\bHen(\Th_{\exists}(F))_{\E}\supseteq\bHen(R)_{\E}$,
and so $R\models\bHen(R)_{\E}$.
Since also $R=R_{\E}$,
we have
$R_{\E}=\bHen(R)_{\E}$,
which proves
	\SAcorrection{{\bf(a)}}%
	.
Next, suppose
	\SAcorrection{{\bf(a)}}
and let $F\models R=\bHen(R)_{\E}$.
Then $F$ is a model of the existential $\Lring$-theory of a henselian field,
thus it is $\ZZ$-large.
\end{proof}

\begin{question}
	What \SAcorrection{more} can we say about $\Lring$-theories $R$ that are fixed points of $\Phi_{A}$?
\end{question}


\section*{Acknowledgements}

This note is founded upon quite a few of the joint works between the Arno Fehm and the author, especially \cite{AF16,AF17,AF26}.
Thus, great thanks are due to Arno for all the conversations and ideas that fed into those works, and directly into this one.
Thanks are also due to an anonymous referee whose careful feedback enriched the manuscript,
to Philip Dittmann for helpful comments on an earlier version,
and to Margarete Ketelsen for discussions and for permission to include Theorem~\ref{thm:Marga}.
Finally: thank you to the DDG seminar, and to each of its organizers, for their enduring committment to Parisian Mathematics, and model-theoretic algebra.

The author
was supported by
the ANR-DFG project ``AKE-PACT'' (ANR-24-CE92-0082)
and
by ``Investissement d'Avenir'' launched by the French Government and implemented by ANR (ANR-18-IdEx-0001) as part of its program ``Emergence''.

\def\bibfont{\footnotesize}
\bibliographystyle{plain}

\begin{thebibliography}{AAA00a}

\bibitem[ADF23]{ADF23}
S.~Anscombe, P.~Dittmann, and A.~Fehm.
\newblock {Axiomatizing the existential theory of $\ps{\FF_{q}}$.}
\newblock {\em Algebra \& Number Theory}, 17-11:2013--2032, 2023.

\bibitem[AF16]{AF16}
S.~Anscombe and A.~Fehm.
\newblock {The existential theory of equicharacteristic henselian valued fields.}
\newblock {\em Algebra \& Number Theory}, 10-3, 665--683, 2016.

\bibitem[AF17]{AF17}
S.~Anscombe and A.~Fehm.
\newblock {Characterizing diophantine henselian valuation rings and valuation ideals.}
\newblock {{\em Proc.~Lond.~Math.~Soc.}, 115:293--322, 2017.}

\bibitem[AF26]{AF26}
\newblock {S.~Anscombe and A.~Fehm.}
\newblock {Interpretations of syntactic fragments of theories of fields.}
\newblock {To appear in {\em Israel Journal of Mathematics}, 2026.}
\newblock {arXiv:2312.17616 [math.LO]}

\bibitem[AJ18]{AJ}
{S.~Anscombe and F.~Jahnke.}
\newblock {Henselianity in the language of rings.}
\newblock {{\em Annals of Pure and Applied Logic}, 169(9):872–895, 2018.}

\bibitem[AJ24]{AJ24}
S.~Anscombe and F.~Jahnke. 
\newblock {Characterizing NIP henselian fields.}
\newblock {{\em J.~Lond.~Math.~Soc.}, (2) 109 (2024), no.~3.}

\bibitem[AK14]{AK14}
S.~Anscombe and J.~Koenigsmann.
\newblock {An existential $\emptyset$-definition of $\ips{\FF_{q}}$ in $\ps{\FF_{q}}$.}
\newblock {{\em J.~Symb.~Log.} 79 (2014), no.~4, 1336--1343.}

\bibitem[Feh15]{Feh15}
A.~Fehm.
\newblock {Existential $\emptyset$-definability of Henselian valuation rings.}
\newblock {{\em J.~Symb.~Log.}, 80 (2015), no.~1, 301--307.}

\bibitem[FJ17]{FJ}
A.~Fehm and F.~Jahnke.
\newblock {Recent progress on definability of Henselian valuations.}
\newblock {In {\em Ordered algebraic structures and related topics}, 135--143.}
\newblock {Contemp.~Math.}, 697 American Mathematical Society, Providence, RI, 2017.

\bibitem[Ket26]{Ketelsen_thesis}
M.~Ketelsen. 
\newblock {Understanding the model theory of henselian valued fields by parts.}
\newblock PhD Thesis, University of Münster, 2026.

\bibitem[KRS24]{KRS}
M.~Ketelsen, S.~Ramello, and P.~Szewczyk
\newblock {Definable henselian valuations in positive residue characteristic.}
\newblock {{\em Journal of Symbolic Logic}, published online, 2024.}

\bibitem[Sze25]{Szewczyk}
P.~Szewczyk.
\newblock {\em On the model theory of valued fields: defect and extremal fields.}
\newblock {Draft PhD thesis, TU Dresden, 2025.}

\end{thebibliography}

\end{document}